\documentclass[a4paper]{article}

\usepackage{amsfonts,amsmath}
\usepackage{latexsym, amsthm}

\newcommand{\R}{\mathbb{R}}
\newcommand{\E}{\mathbb{E}}
\newcommand{\I}{\mathbf{1}}
\newcommand{\var}{{\rm Var}\,}
\newcommand{\conv}{{\rm conv}\,}
\newcommand{\aff}{{\rm aff}\,}
\newcommand{\F}{\mathcal{F}}
\newcommand{\M}{\mathcal{M}}
\newtheorem{theorem}{Theorem}
\newtheorem{lemma}{Lemma}

\date{\today}

\author{I. B\'ar\'any \footnote{Supported by
 Hungarian OTKA grant 60427},
 F. Fodor\footnote{Supported by Hungarian OTKA grants 68398 and 75016 and
by the J\'anos Bolyai Research Scholarship of the Hungarian
Academy of Sciences.}, V. V\'{\i}gh \footnote{Supported by
 Hungarian OTKA grant 75016.}}
\title{Intrinsic volumes of inscribed random polytopes
in smooth convex bodies}

\begin{document}

\maketitle

\begin{abstract}
Let $K$ be a $d$ dimensional convex body with a twice continuously differentiable
boundary and everywhere positive Gauss-Kronecker curvature.
Denote by $K_n$ the convex hull of $n$ points chosen randomly and independently from
$K$ according to the uniform distribution. Matching lower and upper bounds
are obtained for the orders of magnitude of the variances of the $s$-th intrinsic volumes
$V_s(K_n)$ of $K_n$ for $s\in\{1, \ldots, d\}$. Furthermore, strong laws of large numbers
are proved for the intrinsic volumes of $K_n$.
The essential tools are the Economic Cap Covering Theorem of B\'ar\'any and Larman,
and the Efron-Stein jackknife inequality.
\end{abstract}

\section{Notation}
We shall work in $d$-dimensional Euclidean space $\R^d$, with origin $o$, and scalar product
$\langle\cdot, \cdot\rangle$, and induced norm $\|\cdot\|$. The dimension $d$ will be fixed
throughout the paper. We shall not distinguish between the Euclidean space and the underlying
vector space, and we will use the words {\em point} and {\em vector} interchangeably, as we
need them. Points of $\R^d$ are denoted by small-case letters of the roman alphabet, and sets
by capitals. For reals we use either Greek letters or small-case letters. $B^j$ stands for
the $j$-dimensional ball of radius $1$ centered at the origin, $S^{j-1}$ denotes the boundary
of $B^j$ and $\kappa_j$ denotes the volume of $B^j$. Note that any point $x\in
\partial B^j=S^{j-1}$ can be considered as a point of the boundary of $B^j$
and also as an outer normal to $B^j$ at the point $x$. For a point set $T\subset \R^d$, we
denote the \emph{convex hull} of $T$ by $\conv T$ or simply by $[T]$. A compact convex set
$K$ with nonempty interior is called a {\em convex body}.

The {\em intrinsic volumes} $V_s(K)$, $s=0, \ldots , d$ of a convex body $K$
can be introduced as coefficients of the Steiner formula
$$V(K+\lambda B^d)=\sum_{s=0}^{d}\lambda^{d-s}\kappa_{d-s}V_s(K),$$
where $K+\lambda B^d$ is the Minkowski sum of $K$ and $\lambda B^d$ of radius $\lambda\geq
0$. In particular, $V_d$ is the volume functional, $V_0(K)=1$, $V_1$ is proportional to the
mean width and $V_{d-1}$ is a multiple of the surface area. For more information on intrinsic
volumes, see the monographs by Schneider \cite{Sch93}, and Schneider and
Weil \cite{SchW08}. To avoid confusion we use $\lambda_s$ for $s$-dimensional volume (in
particular, $V_d=\lambda_d$).

For a convex body $K$ in $\R^d$, we say that $\partial K$ is $C^k_+$,
for some $k\geq 2$, if $\partial K$ is a $C^k$ manifold and its
Gaussian curvature is positive everywhere.
For a convex body $K$ with $C^2$ boundary and $x\in\partial K$,
we use $\sigma_j(x)$ for the the $j$th normalized elementary symmetric
function of the principal
curvatures of $\partial K$ at $x$. In particular, $\sigma_{d-1}(x)$ is the Gaussian curvature.

We integrate on $G(d,s)$, the {\em Grassmannian manifold} of $s$-dimensional linear subspaces of
$\R^d$. The normalized (and unique) {\em Haar-measure} on $G(d,s)$ is denoted by $\nu_s$
(for details, see \cite{SchW08}). If $L\in G(d,s)$ and $T\subset \R^d$ then we write $T|L$
for the orthogonal projection of $T$ onto $L$. We use $\I(\cdot)$ for the indicator function
of a set. As usual, $\E(\cdot)$ and $\var(\cdot)$ stand for expectation and variance of a
random variable. The notation $\ll$, $\gg$ and $\approx$ are used in the following sense. If
$f(n), g(n): \mathbb{N} \to \R$ are two functions we write $f\ll g$ if there exist a constant
$\gamma$ and a positive number $n_0$ such that we have $f(n)< \gamma g(n)$ for all $n>n_0$.
Furthermore $f\approx g$ if $g\ll f \ll g$. If $n$ is a positive integer, then $[n]$ denotes the
set $\{1,\ldots, n\}$. We write $[n] \choose k$ for the set of all $k$-element subsets of $[n]$.

\section{History and results}
In this paper we consider the following probability model. Let $K$ be a $d$-dimensional
convex body. Select the points $x_1, \ldots , x_n$ randomly and independently from $K$
according to the uniform probability distribution. The density of the uniform distribution
with respect to the Lebesgue measure is the function with the constant value
${\lambda_d(K)}^{-1}$. The convex hull $K_n:=[x_1, \ldots , x_n]$ is a (uniform) random
polytope inscribed in $K$. For a convex body $K$, the expectation $\E_n(V_s)$ of the $s$-th
intrinsic volume of $K_n$ tends to $V_s(K)$ as $n$ tends to infinity, and the shape of the
boundary of $K$ determines the asymptotic behaviour of the random variable
$V_s(K)-\E_n(V_s)$. In this article, we shall prove matching lower and upper bounds for the
order of magnitude of the variance of $V_s(K_n)$ for convex bodies with $C^2_+$ boundary. The
upper bound on the variance will imply a strong law of large numbers for $V_s(K_n)$.

Much effort has been devoted to investigating the properties of various geometric functionals
associated with uniform random polytopes. An up-to-date survey about the current state of
this field can be found in the paper by B\'ar\'any \cite{Bar08}, the book by Schneider and
Weil \cite{SchW08}, and also in the survey by Weil and Wieacker \cite{WW93} from 1993. Here
we only wish to give a brief outline of results that are directly connected with our results.

In particular, the following asymptotic formula is
known about the expectation of intrinsic volumes $V_s(K)$, $s=1,\ldots,d$.
If the boundary $\partial K$ of $K$ is $C^2_+$, then
\begin{equation}
\label{poscurv}
\lim_{n\to\infty}\left(\frac{n}{V(K)}\right)^{\frac2{d+1}}[V_s(K) - \mathbb{E}\, V_{s}(K_n)]=
c_{d,s}\int_{\partial K}\sigma_{d-1}(x)^{\frac1{d+1}}
\sigma_{d-s}(x)\, dx ,
\end{equation}
 with a constant $c_{d,s}>0$ depending only on $d$ and $s$.
Formula (\ref{poscurv}) is due to B\'ar\'any \cite{Bar92}, if $K$ has
$C^3_+$ boundary, and to Reitzner \cite{Rei04} if $K$ has $C^2_+$ boundary.


Until quite recently, very little had been known about the variance of intrinsic volumes of
uniform random polytopes. In 1993, in the survey paper by Weil and Wieacker \cite{WW93}, the
authors state that ``the determination of the variance, for instance, is a major open
problem''. K\"ufer \cite{Kufer} obtained the first result in this direction; he proved the
upper bound $O(n^{-(d+3)/(d+1)})$ for the variance of the missed volume for the
$d$-dimensional unit ball. A major breakthrough was achieved by Reitzner \cite{Rei03},
who proved that, for a convex body $K$ with $C^2_+$ boundary,
\begin{equation}\label{variance-upperbound}
\text{Var }V(K_n)\leq c(K)n^{-(d+3)/(d+1)},
\end{equation}
where the constants $c(K)$ depend on $K$ and the dimension only. The proof
of Reitzner's result rests on the jackknife inequality of Efron and
Stein \cite{ES81}, which we also use in our argument. B\"or\"oczky, Fodor, Reitzner
and V\'\i gh \cite{BFRV09} obtained an upper bound of the same order of magnitude as
in (\ref{variance-upperbound}) for the variance for the mean width of a uniform random polytope
for the case when the mother body has a rolling ball. B\'ar\'any and Reitzner \cite{BR} established an
upper bound for the case when $K$ is a polytope. More precisely, they proved that
\begin{equation}
\text{Var }V(K_n)\leq c(K)\frac{1}{n^2}(\log n)^{d-1},
\end{equation}
where the constant $c(K)$ depends on $K$ and the dimension only.

The above upper bounds imply strong laws of large numbers for the corresponding functionals.

In \cite{Rei05} Reitzner proved matching lower bounds for the variance of the volume functional
for convex bodies with $C^2_+$ boundary, that is,
\begin{equation}
\text{Var }V(K_n)\geq c(K)n^{-(d+3)/(d+1)}.
\end{equation}
These lower bounds were extended by B\'ar\'any and Reitzner \cite{BR} to every convex body in the form
\begin{equation}
\text{Var }V(K_n)\gg \frac{1}{n}V(K(1/n)),
\end{equation}
where $K(1/n)$ is the wet part of $K$ with parameter $1/n$, see Section~3 for details. The
only known lower bound for the variance of an intrinsic volume of a uniform random polytope
other than volume is due to B\"or\"oczky, Fodor, Reitzner and V\'\i gh \cite{BFRV09}. They
established a lower bound with the order of magnitude $n^{-(d+3)/(d+1)}$ for the mean width
of uniform random polytopes for the case when the mother body has a rolling ball.

The variance of the random variable $f_s(K_n)$, which is the number of $s$-dimensional faces
of $K_n$, ($s=0,1,\ldots,d-1$), can be estimated using the above methods as shown
in~\cite{Rei03} and \cite{BR}. Very recently Schreiber and Yukich~\cite{SY} have determined
the variance of $f_0(K_n)$ asymptotically when $K$ is the unit ball, a significant
breakthrough. Hopefully, their methods can work for all $f_s(K_n)$ and $V_s(K_n)$ as well.

In this article we determine the order of magnitude of $\var V_s(K_n)$ when $K=B^d$, the unit
ball.
\begin{theorem}\label{ball}
Let $B^d\subset\R^d$ be the $d$-dimensional unit ball. Let $B^d_n$ be the convex hull of $n$
independent random points chosen from $B^d$ according to the uniform probability
distribution. Then, for $s=1,\ldots, d$,
\begin{equation}
\var V_s(B^d_n)\approx n^{-\frac{d+3}{d+1}}, \text{as } n\to\infty.
\end{equation}
\end{theorem}
The proof of Theorem~\ref{ball} can be extended to smooth convex bodies with $C^2_+$
boundary. All techniques used in the argument for the unit ball apply, with minor variations,
to the case of smooth convex bodies. We give a brief outline of how this can be achieved in
Section~6.
\begin{theorem}\label{main} Let $K\subset \R^d$ be a convex body with
$C^2_+$ boundary. Let $K_n$ be the convex hull of $n$ independent random
points chosen from $K$ according to the uniform probability distribution. Then,
for $s=1,\ldots, d$,
\begin{equation}
\var V_s(K_n)\approx n^{-\frac{d+3}{d+1}}, \text{as } n\to\infty.
\end{equation}
\end{theorem}

The upper bound for the variance of the intrinsic volumes implies a strong law of large numbers
via standard arguments. Thus, we obtain
\begin{theorem}\label{largenum}
Let $K\subset\R^d$ be a convex body with $C^2_+$ boundary and let $K_n$ be the convex
hull of $n$ independent random points from $K$ chosen according to the uniform
distribution. Then for $s=1,\ldots , d$,
\begin{equation}
\lim_{n\to\infty}(V_s(K)-V_s(K_n)) \cdot n^{\frac{2}{d+1}}=
c_{d,s}\cdot \lambda_d(K)^{\frac{2}{d+1}} \int_{\partial K}
\sigma_{d-1}(x)^{\frac{1}{d+1}} \sigma_{d-s}(x)dx
\end{equation}
with probability $1$.
\end{theorem}

The lower bound on the variance can be used to prove the central limit theorem (CLT for
short) for the random variable $V_s(\Pi_n)$. Here $\Pi_n$, the Poisson random polytope, is
similar to the random polytope $K_n$, just for the Poisson polytope, the number of random
points chosen from $K$ is a Poisson distributed random variable with mean $n$. The method of
proving the CLT for this case was introduced by Reitzner~\cite{Rei05} and extended in
B\'ar\'any, Reitzner~\cite{BR}. It works, with more or less straightforward modifications,
for the case of $V_s(\Pi_n)$ when $K$ is either the unit ball, or a $C^2_+$ convex body. The
actual proof is long, technical, and tedious and does not use significant new ideas and is
therefore omitted. Transferring the CLT from the Poisson polytope to the usual random
polytope is often not so simple and was carried out, for $V(K_n)$ and $f_s(K_n)$, by Van Vu
~\cite{Vu06} for smooth convex bodies, and by B\'ar\'any, Reitzner~\cite{BR} for polytopes
using different methods. The same transference for the mixed volumes will, most likely,
require some new method.

\section{Tools}
In this section we describe two statements that will be used in our proof, and
we shall prove a lemma that will be a useful tool for both the lower and upper estimates
of the variance.

If $K$ is a convex body, then a cap of $K$ is a set $C=K\cap H_+$, where $H_+$ is closed half-space.
We define the function $v:K\to \R$ as
$$v(x):=\min\{\lambda_d(K\cap H_+)\, | \, x\in H_+ \text{ and } H_+ \text{ is a closed half-space}\}.$$
The set $K(t)=K(v\leq t)=\{x \in K \,| \, v(x)\leq t\}$ is called the wet part of $K$ with parameter $t>0$.
The remaining part of $K$, namely, $K(v\geq t)=\{x\in K\; | \; v(x\geq t)\}$ is the floating body of $K$ with parameter $t>0$.
One can easily verify that if $K$ is a ball then $\lambda_d(K(t))\approx t^{\frac{d+1}{2}}$.

The following theorem of B\'ar\'any and Larman~\cite{BL88} and B\'ar\'any~\cite{Bar89} plays
a central role in our proof.
\begin{theorem}[Economic Cap Covering]\label{ecocap}
Assume that $K$ is a convex body with unit volume, and $0<t <t_0=(2d)^{-2d}$. Then there are
caps $C_1, \ldots, C_m$ and pairwise disjoint convex sets $C_1',\ldots, C_m'$ such that
$C_i'\subset C_i$ for each $i$, and
\begin{itemize}
\item [(i)] $\bigcup_1^m C_i'\subset K(t) \subset \bigcup_1^m
C_i$,
\item [(ii)] $V_d(C_i') \gg t$ and $V_d(C_i)\ll t$ for each $i$,
\item [(iii)] for each cap $C$ with $C\cap
K(v>t)=\emptyset$ there is a $C_i$ containing C.
\end{itemize}
\end{theorem}

An immediate consequence of this theorem is that $t\cdot m\ll \lambda_d(K(t))\ll m\cdot t$.
For more details and for further
references on the Economical Cap Covering Theorem see \cite{Bar06} and \cite{BL88}.

Our second major tool is the Efron-Stein jackknife inequality (see \cite{E65} and \cite{Rei03}).
If $K_n$ denotes the random polytope inscribed in a convex body $K$ as above,
then the original Efron-Stein theorem readily implies that
\begin{equation}\label{Efronstein}
 \var (V_s(K_n))\leq (n+1) \E(V_s(K_{n+1})-V_s(K_n))^2.
\end{equation}

Finally, we need a simple statement on the measure of special linear subspaces of
$\R^d$. Assume $z \in S^{d-1}$ and $A \in G(d,s)$ are given. Their angle $\angle(z,A)$ is
defined as the minimum of the angles $\angle(z,x)$ for all $x \in A$.

\begin{lemma}\label{szoges}
For fixed $z \in S^{d-1}$ and for small $\alpha>0$, $\nu_s\{A\in G(d,s)\, | \, \angle(z,A)\leq
\alpha\}\approx \alpha^{d-s}$.
\end{lemma}

\begin{proof}
Let $L \in G(d,s-1)$ be a subspace in the orthogonal complement $z^{\perp}$ of the vector $z$.
For every $e \in L^{\perp}\cap S^{d-1}$ with $\angle (e,z)\le \alpha$ the subspace
$A=\,$linspan$\,(L\cup \{e\})$ makes an angle at most $\alpha$ with $z$. Also conversely,
every such subspace $A\in G(d,s)$ can be written in this form. It is not hard to see that the
$\nu_s$-measure of this set is $\approx \alpha^{d-s}$.
\end{proof}

\section{Proof of the lower bound in Theorem~\ref{ball}}
The idea of the proof of the lower bound is similar to those presented in \cite{Rei05} and
\cite{BFRV09}, namely, we define small independent caps, and we show that the variance is
``large'' in each cap. From the properties of the variance the required estimate will follow.

We will use Kubota's formula (see \cite{SchW08}) to represent intrinsic volumes as mean
projections.
\begin{equation}\label{mixvolproj} V_s(K)=c(d,s) \int_{G(d,s)}
\lambda_s (K|L) \nu_s(dL),\end{equation} where $c(d,s)$ is a
constant depending only on $d$ and $s$.

For $x \in S$ and $t\in (0,1)$ we define $H(x,t)=\{z \; |\; \langle z,x \rangle = 1-t\}$ and
we write $x_t=(1-t)x$. Let $C(x,t)$ be the smaller cap cut off from $B^d$ by $H(x,t)$. We
call $x$ the centre of this cap. Clearly $B(x,t)=H(x,t)\cap B^d$ is a $(d-1)$-dimensional
ball centered at the point $x_t$. The radius of $B(x,t)$ is $\sqrt {t(2-t)}$, showing that
\begin{equation}\label{gyokt}
(x_t+\sqrt t B^d)\cap H(x,t) \subset H(x,t)\cap B^d \subset (x_t+\sqrt {2t} B^d) \cap H(x,t).
\end{equation}
This implies that for all $t \in (0,1)$ we have that
\begin{equation}\label{capprop} C(x,t) \subset x+ 2 \sqrt t B^d.\end{equation}
In fact, we will work with $t$ very close to zero (see (\ref{eq:tndef}) below), and all
inequalities with $\ll$ sign below are meant with $t \to 0$.

Next we inscribe a regular $(d-1)$-simplex into $B(x,t)$ whose vertices are the points $w_1,
w_2, \ldots, w_d \in \partial B(x,t)$.  It follows from (\ref{gyokt}) that the simplex
$[w_1,\ldots, w_d]$ contains the $(d-1)$-ball of radius $\sqrt t /d$ centered at the point
$x_t$. Set $w_0=x$. Then $\Delta=[w_0, w_1, \ldots, w_d]$ is a $d$-dimensional simplex
inscribed in $C(x,t)$.

For all $j=0,1,\ldots, d$ we define
$$\Delta_j=\Delta_j(x,t)=w_j+\frac{1}{4d}([w_0, w_1, \ldots,
w_d]-w_j).$$ $\Delta_j$ is a homothetic copy of $\Delta$ with centre $w_j$ and factor
$1/(4d)$. It readily follows from (\ref{gyokt}) that $V(\Delta_j(x,t))\approx t^{\frac{d+1}{2}}$.
Choose a point $z_j$ in each $\Delta_j(x,t)$, and define
$$\Sigma_1(x,t)=S^{d-1} \cap \left (x + \frac{\sqrt t}{8}B^d \right )$$ and
$$\Sigma_2(x,t)=S^{d-1} \cap \left (x + 2d \sqrt tB^d \right ).$$\\ Set $\Delta(z)=[z_0,\ldots,z_d]$
and write $N$ for the cone of outer normals to $\Delta(z)$ at vertex $z_0$. We claim
that
\begin{equation}\label{szigmakok}
\Sigma_1(x,t)\subset S^{d-1}\cap N \subset \Sigma_2(x,t).
\end{equation}
To prove (\ref{szigmakok}) pick an arbitrary $v\in S^{d-1}$ such that
$\langle v,x\rangle=0$. From the definition of $\Delta_j$ and from
(\ref{gyokt}) we obtain that
$$\frac{\sqrt t}{2d}\leq h_{\Delta(z)}(v)-\langle z_0,
v \rangle\leq 2 \sqrt t,$$  where $h_{\Delta(z)}(.)$ is the support function of $\Delta(z)$.
Similarly
$$\frac t2\leq \langle z_0, x \rangle - h_{\Delta(z)}(x) \leq t.$$

From these we deduce that the ``extremal'' element $u$ of the normal cone $N$ in the
direction $v$ ($u=\langle u,x \rangle x+ \langle u,v \rangle v$) satisfies
$$\frac{\sqrt t}{4}\leq \tan (\angle(u,x)) \leq 2d\sqrt t,$$ and so
the claim (\ref{szigmakok}) follows.

(\ref{szigmakok}) can be dualized:
\begin{equation}\label{dualszigma}
\Sigma_2^*(x,t)\subset \{\lambda(y-z_0)\; | \; \lambda\geq 0, y\in
[z_0,z_1, \ldots, z_d]\}\subset \Sigma_1^*(x,t),
\end{equation}
where $\Sigma_j^*(x,t)=\{y \; |\; \langle y, u \rangle \leq 0,
\forall u\in \Sigma_j(x,t) \}$ is the usual dual cone of $\Sigma_j$.
Note that (\ref{szigmakok}) also implies that there exists an
absolute constant $\gamma$, such that
\begin{equation}
\label{capind} B^{d}\backslash C(x,\gamma t)\subset z_0+\Sigma^*_2(x,t).
\end{equation}

Now fix $x$, $t$ and $z_j \in \Delta_j(x,t)$ for $j=1,\ldots, d$. We
write $F=[z_1, \ldots, z_d]$. Define the function $\hat V_s :
\Delta_0(x,t) \to \mathbb{R}$ as follows
$$\hat V_s(z_0)= \int_{L
\in G(d,s),\\ L\cap \Sigma_2\neq \emptyset} \lambda_s([z_0,F] | L) \nu_s(dL).$$ $\hat V_s$
clearly depends on $F$, if we want to emphasize this dependence, then we write $\hat
V_s(z_0;F)$.
\begin{lemma}\label{alsokulcs}
If $Z$ is a random point chosen uniformly from $\Delta_0(x,t)$
then $${\rm Var} \;\hat V_s(Z)\gg t^{d+1}.$$
\end{lemma}
\begin{proof}
Let $w$ be the centroid of the facet of $\Delta_0(x,t)$ opposite to
$x$, let $w_1=\frac23\,x+\frac13\,w$ and
$w_2=\frac13\,x+\frac23\,w$. In addition we define
\begin{eqnarray*}
\Psi_1&=&(w_1-\Sigma_2^*(x,t))\cap \Delta_0(x,t),\\
\Psi_2&=&(w_2+\Sigma_2^*(x,t))\cap \Delta_0(x,t).
\end{eqnarray*}
In particular there exists some constant $\gamma_0>0$ such that
\begin{equation}\label{nagypszi}
 V(\Psi_j)\geq \gamma_0 V(\Delta_0(x,t)),
\end{equation}
and for any $Z_1\in\Psi_1$ and  $Z_2\in\Psi_2$ we have
$[Z_2,z_1,\ldots,z_d]\subset [Z_1,z_1,\ldots,z_d]$.

Fix $L \in G(d,s)$ such that $L\cap \Sigma_2\neq \emptyset$, and
choose an orthonormal basis $e_1, \ldots, e_s$ in $L$, such that
$e_1 \in L\cap \Sigma_2$.

Consider the closed (positive) half-space given by $w_2$ and $e_1$:
$H_1^+=\{y\; | \; \langle y, e_1 \rangle\geq \langle w_2, e_1
\rangle\}$, and the set $G=H_1^+ \cap (Z_1+\Sigma_2^*(x,t))$.
Clearly, $G\subset [F,Z_1]$ and $G\cap [F,Z_2]\subset\{w_2\}$. From
these it follows that
$$\lambda_s([F,Z_1]|L)-\lambda_s([F,Z_2]|L)\geq \lambda_s(G|L).$$

One can see that
$$\lambda_s(G|L)\gg t \cdot \sqrt t^{(s-1)}=t^{\frac{s+1}2},$$ hence
$$ \hat V_s(Z_1)-\hat V_s(Z_2)\gg t^{\frac{s+1}2} \cdot \nu_s (\{L \; |\; L\cap \Sigma_2\neq \emptyset\}). $$

By the definition of $\Sigma_2$, $L\cap \Sigma_2\neq \emptyset$ is equivalent to $L\cap (x +
2d \sqrt tB) \neq \emptyset$. Lemma \ref{szoges} yields that
$$\hat V_s(Z_1)-\hat V_s(Z_2)\gg t^{\frac{d+1}2}.$$ Finally, we obtain
\begin{eqnarray*}
\var \hat V_s(Z) & = & \frac12 \E(\hat V_s(Z_1)-\hat V_s(Z_2))^2\geq\\
& \geq &\frac12 \E[(\hat V_s(Z_1)-\hat V_s(Z_2))^2 \I(Z_1\in\Psi_1,
Z_2\in \Psi_2)]\gg\\
&\gg & t^{d+1} \E[\I(Z_1\in\Psi_1, Z_2\in \Psi_2)]\gg t^{d+1},
\end{eqnarray*}
where the last inequality follows from (\ref{nagypszi}).
\end{proof}

It is sufficient to prove the  lower bound for large enough $n$. We
fix
\begin{equation}\label{eq:tndef}
t_n=n^{-\frac2{d+1}},
\end{equation}
and hence $V(C(x,t_n))\approx 1/n$ for all $x\in S^{d-1}$. We choose a
maximal family of points $y_1,\ldots,y_m\in S^{d-1}$ such that for $i\neq
j$, we have
$$
\|y_i-y_j\|\geq 2\sqrt{\gamma}\sqrt{t_n},
$$
This condition implies that the caps $C(y_j,\gamma T_n)$ ($j \in [m]$) are disjoint. One can see that
\begin{equation}
\label{msize} m\gg n^{\frac{d-1}{d+1}}.
\end{equation}

For each $j \in [m]$ we construct the simplex $\Delta (y_j,t_n)$ in the cap $C(y_j,t_n)$ and
for each $i=0,1,\dots,d$ we construct the corresponding small simplices $\Delta_i(y_j,t_n)$.
For $j\in [m]$, let $A_j$ denote the event that each $\Delta_i(y_j,t_n)$, $i=0,\ldots,d$
contains exactly one random point out of $x_1,\ldots,x_n$, and $C(y_j,\gamma t_n)$ contains
no other random point. We note that the definition of $\Delta_i$, (\ref{capprop}) and
(\ref{capind}) yield that for $i=0,\ldots,d$, we have
$$
V(\Delta_i(y_j,t_n))\gg 1/n \mbox{ \ and \ }V(C(y_j,\gamma t_n))\ll
1/n.
$$
Thus for $j=1,\ldots,m$, we have
\begin{equation}
\label{Ajprob} \mathbb{P}(A_j)\gg \binom {n} {d+1}
\left(\frac{1}n\right)^{d+1}\left(1-\frac{1}n\right)^{n-d-1}\gg 1.
\end{equation}

If $A_j$ holds then we write $Z_j$ to denote the random point in
$\Delta_0(y_j,t_n)$, and $F_j$ to denote the convex hull of the
random points in $\Delta_i(y_j,t_n)$ for $i=1,\ldots,d$. If $J \subset [m]$ and $A_j$
holds for all $j \in J$, then the random variables $\hat
V_s(Z_j;F_j)$ $j \in J$ are independent according to (\ref{capind}).

We next introduce the sigma algebra $\mathcal F$ that keeps track of
everything except the location of $Z_j \in \Delta_0(y_j, t_n)$ for
which $A_j$ occurs. We decompose the variance by conditioning on
$\mathcal F$:
\begin{eqnarray*}
{\var}V_s(K_n) &=& \E\, {\var}( V_s(K_n)\, \vert\, \mathcal F) + {\var} \, \E (V_s(K_n) \vert \, \mathcal F ) \\
&\geq & \E( {\var} V_s(K_n)\, \vert\, \mathcal F) .
\end{eqnarray*}
The independence structure mentioned above implies that
\begin{eqnarray*}
{\var} (V_s(K_n) \, \vert\, \mathcal F) &=& \sum_{  \I(A_j)=1}
{\var} _{Z_{j}} V_s(K_n)
\\ &=&
\sum_{  \I(A_j)=1}  {\var}_{Z_j} \hat V_s (Z_j;F_j)
\end{eqnarray*}
where the variance is taken with respect to the random variable $Z_{j} \in \Delta_0(y_j,
t_n)$, and we sum over all $j=1, \dots ,m$ with $\I(A_j)=1$. Combining this with Lemma
\ref{alsokulcs}, (\ref{eq:tndef}), (\ref{msize}) and with (\ref{Ajprob}) implies
\begin{eqnarray*}
{\var} V_s(K_n) &\gg & \E \left( \sum_{ \I(A_j)=1} t_n^{d+1} \right)
\gg n^{-2} \E \left( \sum_{j=1}^m I(A_j) \right)\\
& \gg &    n^{-2} m  \gg n^{- \frac {d+3}{d+1}}.
\end{eqnarray*}

\section{Proof of the upper bound in Theorem~\ref{ball}}
Now $K=B^d$ is the unit ball and $K_n$ is the corresponding random polytope.

Let $T_n$ be the event that the floating body $K(v \ge (c\log n)/n V(K))$ is contained in
$K_n$. Here $c=c_d$ is a large constant to be specified soon. We write $T_n^c$ for the
complement of $T_n$. We are going to use the main result of \cite{BD} saying that there is a
constant $\delta$ depending only on $d$ such that $T_n^c$ occurs with probability $n^{-\delta
c}$.

We use the Efron-Stein jackknife inequality (\ref{Efronstein}) and Kubota's formula
(\ref{mixvolproj}):

\begin{eqnarray*}
\var (V_s(K_n)) &\ll&  n\cdot \E(V_s(K_{n+1})-V_s(K_n))^2\\
          &=& n\cdot \E[(V_s(K_{n+1})-V_s(K_n))^2\I(T_n)]\\
          &+& n\cdot \E[(V_s(K_{n+1})-V_s(K_n))^2\I(T_n^c)].
\end{eqnarray*}
The second term here is very small if the constant $c$ is chosen large enough because
$(V_s(K_{n+1})-V_s(K_n))^2 \le V_s(K_{n+1})^2\le V_s(K)^2$ and $\E(\I(T_n^c)) \le n^{-\delta
c}$. We choose $c=c_d$ so large that the second term is smaller than the lower bound in
Theorem~\ref{ball} proved in the previous section. So we concentrate on the first term:

\begin{align}\label{alap}
\var (V_s(K_n)) &\ll  n\cdot \E[(V_s(K_{n+1})-V_s(K_n))^2\I(T_n)]\notag\\
&\ll n\cdot \E \left [ \left (\int_{G(d,s)} \lambda_s(K_{n+1}\backslash K_n|A) \nu_s(dA) \right)\right.\times\notag\\
&\phantom{xxxxx}\times \left.\left(\int_{G(d,s)} \lambda_s(K_{n+1}\backslash K_n|B) \nu_s(dB)\I(T_n) \right )\right ]\notag\\
&\ll n\cdot \E \int_{G(d,s)}\int_{G(d,s)}
\lambda_s(K_{n+1}\backslash K_n|A) \times\notag\\
&\phantom{xxxxx}\times\lambda_s(K_{n+1}\backslash K_n|B)\I(T_n) \nu_s(dA)\nu_s(dB).
\end{align}

Note that the set $(K_{n+1}\backslash K_n)|A$ is either empty (if $x_{n+1}|A\in K_n|A$) or it is
the union of several internally disjoint simplices which are the convex hull of $x_{n+1}|A$
and those facets of $K_n|A$ that can be seen from $x_{n+1}|A$. For the index set
$I=\{i_1,\ldots , i_s \} \subset \{1,\ldots, n\}$, let $F_I=[x_{i_1}, \ldots, x_{i_s}]$,
which is an $(s-1)$-dimensional simplex with probability $1$. Clearly $F_I|A$ is also an
$(s-1)$-simplex with probability $1$. The affine hull of $F_I$ is denoted by $\aff F_I$ and
similarly the affine hull of $F_I|A$ is by $\aff(F_I|A)$. Furthermore, let $H_0(F_I,A)$ be
the closed half-space (in $\R^d$) delimited by the hyperplane $A^{\perp}+\aff F_I$ that
contains $o$, and $H_+(F_I,A)$ the other one. Similarly, we use $H_0(F_I|A)$ and $H_+(F_I|A)$
for the corresponding $s$-dimensional half-spaces in $A$. Now, we introduce the notation $\F
(A)$ for the set of ($(s-1)$-dimensional) facets of $K_n|A$ that can be seen from
$x_{n+1}|A$.
\begin{align*} \F (A) &=\{F_I|A \; : \; K_n|A \subset H_0(F_I|A), x_{n+1}|A\in H_+(F_I|A),\\
&\phantom{xxxx}I=\{i_1, \ldots, i_s\}\subset \{1, \ldots, n\}\}.\end{align*}

Of course $\F(A)$ depends on $x_1,\ldots,x_n$ and $x_{n+1}$ as well but we suppress this
dependence in the notation. We continue by estimating the right hand side of (\ref{alap}).

\begin{align*}
&\leq n\cdot\E \left [ \int_{G(d,s)}\int_{G(d,s)}
\lambda_s(K_{n+1}\backslash K_n|A) \lambda_s(K_{n+1}\backslash
K_n|B) \nu_s(dA)\nu_s(dB)\I(T_n)\right ]\\
&\ll\frac{n}{\kappa_d^{n+1}} \int_{K} \ldots \int_{K} \int_{G(d,s)}\int_{G(d,s)} \left
(\sum_{F\in
\F(A)} \lambda_s([x_{n+1}|A, F])\right ) \times \\
&\times  \left (\sum_{F'\in \F(B)} \lambda_s([x_{n+1}|B, F'])\I(T_n)\right ) \nu_s(dA)
\nu_s(dB)dx_1\cdots dx_{n+1}.
\end{align*}
By changing the order of integration and extending integration
over all index sets $I, J\in {[n]\choose s}$, we obtain the following.
\begin{align*}
&=\frac{n}{\kappa_d^{n+1}}  \int_{G(d,s)}\int_{G(d,s)} \int_{K^{n+1}} \left
(\sum_{I} \I(F_I|A\in \F (A))\lambda_s([F_I, x_{n+1}]|A)\right ) \times\\
&\times  \left (\sum_{J} \I(F_J|B\in \F(B) )\lambda_s([F_J, x_{n+1}]|B)\I(T_n)\right )
dx_1\cdots dx_{n+1}\nu_s(dA) \nu_s(dB).
\end{align*}
We use the following notations. Let
$$C_s(I,A)=H_+(F_I|A)\cap B^d,$$ which is, in fact, a subset of the unit ball in the subspace $A$ and
$$C_d(I,A)=(H_+(F_I|A)+A^\perp)\cap B^d.$$
For the volumes of these caps we use $V_s(I,A)=\lambda_s(C_s(I,A))$ and
$V_d(I,A)=\lambda_d(C_d(I,A))$. Now we are going to estimate these integrals from above using
the fact the simplices $[F_I, x_{n+1}]|A$ and $[F_J, x_{n+1}]|B$ are contained in the
associated caps $C_s(I,A)$ and $C_s(J,B)$, respectively.
\begin{align}\label{seged1}
\ll\frac{n}{\kappa_d^{n+1}}
&\int_{G(d,s)}\int_{G(d,s)} \sum_{I} \sum_{J} \int_{(B^d)^{n+1}}
\I(F_I|A\in \F (A)) V_s(I,A) \times\notag\\
&\times\I(F_J|B\in \F(B) ) V_s(J,B) \I(T_n)dx_1\cdots dx_{n+1}\nu_s(dA) \nu_s(dB).
\end{align}
The summation extends over all $s$-tuples $I$ and $J$, so $I$ and $J$ may have nonempty
intersection. If we fix the size of $I\cap J$ to be $k$, say, then the corresponding terms in
the sum are clearly independent of the particular choice of $i_1, \ldots, i_s$ and $j_1,
\ldots, j_s$. For any given $k\in \{0,1, \ldots, s\}$ let $I=\{1,\dots,s\}$ and
$J=\{s-k+1,\dots2s-k\}$ and set $F=\conv\{x_i:i\in I\}$ and $G=\conv\{x_j: j\in J\}$. Thus
$I$ and $J$, and consequently $F$ and $G$ depend on $k$, but this is not shown in the
notation. We can estimate (\ref{seged1}) from above by
\begin{align}\label{seged2}
&\leq\frac{n}{\kappa_d^{n+1}} \sum_{k=0}^s {n \choose s}
{s \choose k} {{n-s} \choose {s-k}} \int_{G(d,s)}\int_{G(d,s)}
\int_{(B^d)^{n+1}} \I(F|A\in \F(A)) \times\notag\\
&\times V_s(I,A)\I(G|B\in \F(B) )V_s(J,B) \I(T_n)dx_1\cdots dx_{n+1}\nu_s(dA) \nu_s(dB).
\end{align}
Since the integrand is symmetric, we may restrict summation to those pairs of $F$ and $G$
where $V_s(I,A)\geq V_s(J,B)$, or equivalently, $V_d(I,A)\geq V_d(J,B)$, at
the price of a factor $2$. Thus, we can estimate (\ref{seged2}) from above by
\begin{align}\label{borzaszto}
\leq&\sum_{k=0}^s \frac{2n}{\kappa_d^{n+1}}
 {n \choose s} {s \choose k} {{n-s} \choose {s-k}}
\int_{G(d,s)}\int_{G(d,s)} \int_{(B^d)^{n+1}}\I(F|A\in \F (A))
 \times\notag\\
& \times V_s(I,A)\I(G|B\in \F(B))V_s(J,B)\I(V_s(I,A)\geq V_s(J,B)) \times\notag\\
& \times\I(T_n)dx_1\cdots dx_{n+1} \nu_s(dA) \nu_s(dB).
\end{align}
Let $\Sigma_k$ denote the $k$th term in this sum, $k=0,\ldots,s$. We are going to estimate
$\Sigma_k$ for each fixed $k$.

We first remove $\I(G|B\in \F(B) )$ from the integrand in $\Sigma_k$ which clearly increases
the integral. We then multiply the integrand by $\I(C_d(I,A)\cap C_d(J,B)\neq \emptyset)$.
This does not change the integral since the sets $C_d(I,A)$ and $C_d(J,B)$ have at least the
point $x_{n+1}$ in common. Thus we obtain the following
\begin{align}\label{seged3}
\Sigma_k \ll &\frac{n^{2s-k+1}}{\kappa_d^{n+1}} \int_{G(d,s)}\int_{G(d,s)} \int_{(B^d)^{n+1}}
\I(F|A\in \F (A))V_s(I,A)\times\notag\\
&\times \I(C_d(I,A)\cap
C_d(J,B)\neq \emptyset) V_s(J,B)\I(V_s(I,A)\geq V_s(J,B))\times\notag\\
&\times  \I(T_n) dx_1\cdots dx_{n+1} \nu_s(dA) \nu_s(dB).
\end{align}

Now, if $F|A\in\F(A)$, then $x_{2s-k+1}, \ldots, x_n$ are all contained in $H_0(F,A)$ and
$x_{n+1}$ is contained in $H_+(F,A)$ because, under condition $T_n$, $C_d(I,A)$ is the
smaller cap cut off from $B^d$ by the hyperplane $A^{\perp}+\aff F$ and $o \in K_n$. We
integrate with respect to $x_{2s-k+1}, \ldots, x_n, x_{n+1}$, and the condition $T_n$ is
replaced by the condition $W_n$ saying that $V_d(I,A)\le (c\log n)/n V(B^d)$.

\begin{align}\label{seged4}
\Sigma_k \ll &\frac{n^{2s-k+1}}{\kappa_d^{n+1}} \int_{G(d,s)}\int_{G(d,s)}
\int_{(B^d)^{2s-k}}
(\kappa_d-V_+(F,A))^{n-2s+k}V_d(I,A)\times\notag\\
&\times  V_s(I,A)\I(C_d(I,A)\cap
C_d(J,B)\neq \emptyset)  V_s(J,B)\times\notag\\
&\times \I(V_d(I,A)\geq V_d(J,B)) \I(W_n)dx_1\cdots dx_{2s-k}
\nu_s(dA) \nu_s(dB)\notag\\
\ll &n^{2s-k+1} \int_{G(d,s)}\int_{G(d,s)} \int_{(B^d)^{2s-k}}
(1-V_d(I,A)/\kappa_d)^{n-2s+k} \times\notag\\
&\times V_d(I,A) V_s(I,A)\I(C_d(I,A)\cap
C_d(J,B)\neq \emptyset)  V_s(J,B)\times\notag\\
&\times \I(V_d(I,A)\geq V_d(J,B)) \I(W_n)dx_1\cdots dx_{2s-k} \nu_s(dA) \nu_s(dB).
\end{align}

In the next step, we integrate with respect to the variables $x_i$, $i \in J$.
\begin{eqnarray}\label{pontokra}
\nonumber \int_{(B^d)^{s-k}} \I(C_d(I,A)\cap C_d(J,B)\neq
\emptyset)\I(V_d(I,A)\geq V_d(J,B))\times\\
\times V_+(2,B,s) \I(W_n)dx_{s+1}\ldots dx_{2s-k}.
\end{eqnarray}

Since we assume that $V_d(I,A)\geq V_d(J,B)$, the height of the cap $C_d(I,A)$ is at least that of
$C_d(J,B)$. The condition $C_d(I,A)\cap C_d(J,B)\neq \emptyset$ implies that there is a
constant $\beta$ such that $C_d(J,B)$ in contained in $\beta C_d(I,A)$, where $\beta
C_d(I,A)$ is an enlarged homothetic copy of $C_d(I,A)$, where the centre of homothety is $z
\in \partial B^d$, the centre of the cap $C_d(I,A)$ (cf \cite{Bar06}). Thus,

\begin{eqnarray}\label{pontokbecs}
(\ref{pontokra}) \leq \beta^{d(s-k)} V_d(I,A)^{s-k} V_s(J,B)\ll V_d(I,A)^{s-k} V_s(I,A).
\end{eqnarray}

The conditions $C_d(I,A)\cap C_d(J,B)\neq \emptyset$  and $V_d(I,A)\geq V_d(J,B)$ can only be
satisfied if the angle, $\angle(z,B)$, between the vector $z$ and the subspace $B$ is not
larger than $2\alpha$, where $\alpha$ is the central angle of the cap $C_d(I,A)$. One can
easily verify that
\begin{equation}\label{alfakicsi}
\alpha\le  b_d V_d(I,A)^{1/(d+1)},
\end{equation}
where $b_d$ is a constant depending only on $d$. Using this condition on the mutual positions
of $z$ and $B$ together with (\ref{pontokbecs}) we obtain that

\begin{align}\label{seged5}
(\ref{seged4})\ll &n^{2s-k+1}\int_{G(d,s)}\int_{G(d,s)}
\int_{(B^d)^{s}}
(1-V_d(I,A)/\kappa_d)^{n-2s+k}V_d(I,A)^{s-k+1} \times\notag\\
&\times  V_s(I,A)^2 \I(\angle(z,B)\leq 2b_d V_d(I,A)^{1/(d+1)}) \times \notag\\
&\times \I(W_n)dx_1\cdots dx_{s} \nu_s(dA) \nu_s(dB).
\end{align}

We fix now $A\in G(d,s)$ and estimate
\begin{equation}\label{belso}
\int_{(B^d)^{s}} (1-V_d(I,A)/\kappa_d)^{n-2s+k} V_d(I,A)^{s-k+1} V_s(I,A)^2 \I(W_n)dx_1\cdots
dx_{s}
\end{equation}
We are going to use the Economic Cap Covering Theorem. Because of the condition $W_n$, every
cap $C_d(I,A)$ has volume at most $(c\log n)/n \kappa_d$. Let $h$ be a (positive) integer
with $2^{-h}\leq \frac{c \ln n}{n}$. For each such $h$, let $\M_h$ be a collection of caps
$\{C_1, \ldots, C_{m(h)}\}$ forming the economic cap covering of the wet part of $B^s=B^d|A$
with $t=(\kappa_d 2^{-h})^{\frac{s+1}{d+1}}$ (we suppose that $n$ is so large, that the
theorem works). Each such cap $C_i$ is the projection of a $d$-dimensional cap $C_i(A)$ from
$B^d$ to $A$. Since the heights of $C_i$ and $C_i(A)$ are equal, we have that
$\lambda_d(C_i(A))\ll \kappa_d 2^{-h}$. Consider an arbitrary $(x_1, \ldots, x_s)$ with the
corresponding $C_d(I,A)$ having volume at most $(c\log n)/n \kappa_d$, and associate with
$(x_1, \ldots, x_s)$ the maximal $h$ such that for some $C_i\in \M_h$, $C_s(I,A)\subset C_i$.
Such an $h$ clearly exists. It follows that
$$V_s(I,A)\leq \lambda_s(C_i)\ll 2^{-h\frac{s+1}{d+1}}$$
and $$V_d(I,A)\leq \lambda_d(C_i(A))\ll 2^{-h}.$$ On the other hand, by the maximality of
$h$,
$$V_s(I,A)\geq (\kappa_d 2^{-(h+1)})^{\frac{s+1}{d+1}}$$
and consequently $$V_d(I,A)/\kappa_d \geq 2^{-(h+1)}.$$ Now we shall integrate over $(B^d)^s$ under
condition $W_n$ by integrating each $(x_1,\ldots, x_s)$ on its associated $C_i(A)$, or more
precisely on $(C_i(A))^s$. The integrand in (\ref{belso}) can be estimated as
\begin{eqnarray*}(1-V_d(I,A)/\kappa_d)^{n-2s+k} V_d(I,A)^{s-k+1}
V_s(I,A)^2\\
\ll (1-2^{-(h+1)})^{n-2s+k}\cdot 2^{-h(s-k+1)}\cdot
2^{-2h\frac{s+1}{d+1}}.
\end{eqnarray*}
Thus the integral on $(C_i(A))^s$ ($C_i(A)\in \M_h$) is bounded by
\begin{eqnarray}\label{integrandh}
\nonumber \exp(-(n-2s+k)2^{-h-1}))2^{-h(s+k-1)}2^{-2h\frac{s+1}{d+1}}(V_d(C_i(A)))^s\\
\ll
\exp(-(n-2s+k)2^{-h-1}))2^{-h(s+k-1)}2^{-2h\frac{s+1}{d+1}}2^{-hs}.
\end{eqnarray}

Now we return to (\ref{seged5}). In order to estimate the integral, we still need the number
of the elements $|\M_h|$ of $\M_h$. The volume of the wet part of $B^s$ with parameter
$2^{-h\frac{s+1}{d+1}}$ is $\lambda_s(B^s(2^{-h\frac{s+1}{d+1}}))\approx 2^{\frac{-2h}{d+1}}$
(the $\approx$ notation makes sense, since $h\to \infty$ as $n\to \infty$). It readily
follows that
$$|\M_h|\ll \frac{2^{-2h\frac{1}{d+1}}}{2^{-h\frac{s+1}{d+1}}}=2^{\frac{h(s-1)}{d+1}}.$$

 Keeping in mind the condition $\angle(z,B)\leq 2b_d V_d(I,A)^{1/(d+1)}$ and applying Lemma \ref{szoges} and
(\ref{integrandh}), we obtain with $h_0=\left \lfloor \frac{c \ln n}{n} \right \rfloor$,
\begin{eqnarray}
\nonumber && \int_{G(d,s)} \int_{(B^d)^{s}} (1-V_d(I,A)/\kappa_d)^{n-2s+k}
V_d(I,A)^{s-k+1} V_s(I,A)^2 \times\\
\nonumber &\times &\I(\angle(z,B)\leq 2b_d V_d(I,A)^{1/(d+1)})
dx_1\cdots dx_{s} \nu_s(dB)\\
\nonumber &\ll & \sum_{h=h_0}^{\infty} \exp(-(n-2s+k)2^{-h-1}))2^{-h(s+k-1)}2^{-2h\frac{s+1}{d+1}}2^{-hs}\times \\
\nonumber & \times & |\M_h| \nu_s(\{B \, | \, \angle(z,B)< 2b_d 2^{\frac{-h}{d+1}}\})\\
\nonumber & \ll & \sum_{h=h_0}^{\infty} \exp(-(n-2s+k)2^{-h-1}))2^{-h(s+k-1)}2^{-2h\frac{s+1}{d+1}}2^{-hs} 2^{\frac{h(s-1)}{d+1}} 2^{\frac{-h(d-s)}{d+1}}\\
&=& \sum_{h=h_0}^{\infty} \exp(-(n-2s+k)2^{-h+1}))2^{-h[(2s-k+1)+\frac{d+3}{d+1}]}.
\end{eqnarray}

Now, we divide the sum in (\ref{integrandh}) into two parts. First, let $h_1$ be defined by $$2^{-h_1}\leq \frac1n < 2^{-h_1+1}.$$ Since in this case $\exp(-(n-2s+k)2^{-h-1}))$ is smaller than $1$, it follows that
\begin{eqnarray}\label{farok}
 \nonumber \sum_{h=h_1}^{\infty} \exp(-(n-2s+k)2^{-h-1}))2^{-h[(2s-k+1)+\frac{d+3}{d+1}]}\\
\leq\sum_{h=h_1}^{\infty} 2^{-h[(2s-k+1)+\frac{d+3}{d+1}]}\ll n^{-2s+k-1}n^{-\frac{d+3}{d+1}}.
\end{eqnarray}
For the other part, when $h_0\leq h<h_1$, we let $\ell=h_1-h$. Then $\ell$ runs from $1$ to
$\ell_1=h_1-h_0$.
\begin{eqnarray}\label{torzs}
 \nonumber & & \sum_{h=h_0}^{h_1-1} \exp(-(n-2s+k)2^{-h-1}))2^{-h[(2s-k+1)+\frac{d+3}{d+1}]}\\
\nonumber & \leq & \sum_{\ell=1}^{\ell_1} \exp(-(n-2s+k)2^{-h_1+\ell-1}))2^{-(h_1-\ell)
[(2s-k+1)+\frac{d+3}{d+1}]}\\
\nonumber & \ll & \sum_{\ell=1}^{\ell_1} \exp(-(n-2s+k)2^{-h_1+\ell-1})) n^{-(2s-k+1)}
2^{\ell(2s-k+1)}n^{-\frac{d+3}{d+1}}2^{\ell \frac{d+3}{d+1}}\\
\nonumber & \ll & n^{-(2s-k+1)}n^{-\frac{d+3}{d+1}} \sum_{\ell=1}^{\infty} \exp (-2^\ell)
\cdot 2^{\ell[(2s-k+1)+\frac{d+3}{d+1}]}\\
& \ll & n^{-(2s-k+1)}n^{-\frac{d+3}{d+1}} \sum_{j=1}^{\infty} \exp (-j) j^{4d}\ll n^{-(2s-k+1)}
n^{-\frac{d+3}{d+1}}.
\end{eqnarray}

Now, putting (\ref{farok}) and (\ref{torzs}) back to (\ref{seged5}), we get that
\begin{eqnarray*}
\Sigma_k \ll n^{2s-k+1}\int_{G(d,s)} n^{-(2s-k+1)}n^{-\frac{d+3}{d+1}}\nu_s(dA)\ll
n^{-\frac{d+3}{d+1}}.
\end{eqnarray*}
Summing this for all $k=0,\ldots,s$ proves the upper bound in Theorem \ref{main}.

\section{Outline of the proof of Theorem~\ref{main}}

In this section we will give a brief outline how to modify the proof of Theorem
\ref{ball} so that it applies to convex bodies with $C^2_+$ boundary. We chose to give the
detailed proof only for the unit ball, because all the major ideas appear in that case and it
is naturally easier to follow. The proof in the general case uses the same tools as the proof
for the unit ball combined with some well-known estimates about convex bodies. We make only a
few remarks about the proof and leave the details for the interested reader.

Since $K$ is compact there exist a global upper bound $\gamma$ and a global
lower bound $\Gamma$ on the principal curvatures of $\partial K$.
We also know that for every $x\in \partial K$, there exists a unique outer unit normal
$u_x$ to $K$ at the point $x$. We define the cap $C(x,t)$ such that it is
cut off by the hyperplane $H(x,t):=\{y\, |\, \langle y, u_x \rangle =\langle x, u_x\rangle -t\}$.
It readily follows that (\ref{gyokt}) remains true in the following form:
\begin{equation}\label{gyoktmod}
(x_t+\gamma_1 \sqrt t B^{d} )\cap H(x,t) \subset H(x,t)\cap K \subset (x_t+\gamma_2\cdot \sqrt t B^{d}) \cap H(x,t),
\end{equation}
where the constants $\gamma_1$ and $\gamma_2$ depend on $\gamma$ and $\Gamma$. These estimates
yield that the simplices used in the proof can be defined in the same way as in the case of
the unit ball, and they have the same size $\approx t^{\frac{d+1}{2}}$. We also need a slight
modification of the definitions of  $\Sigma_1$ and $\Sigma_2$:
$$\Sigma_1(x,t)=S^{d-1} \cap \left (u_x +
\frac{\sqrt{\gamma t}}{8}B^d \right )$$ and $$\Sigma_2(x,t)=S \cap \left (u_x + 2d
\sqrt{\Gamma t}B^d \right ).$$ From this point, the steps of the proof can be followed
without complications. For the details of a similar argument see \cite{BR} or \cite{BFRV09}.

For the proof of the upper bound, we need that all projected images of $K$ to have $C^2_+$
boundary, furthermore, we can choose $\gamma$ and $\Gamma$ in such a way that they are not
only upper and lower bounds of the principal curvatures of $K$ but also for all
($s$-dimensional) projections of $K$. These facts yield that the volume of a cap of height
$t$ is $\approx t^{\frac{i+1}{2}}$, where $i$ is the dimension of the cap (in the proof $d$
or $s$, respectively). From (\ref{gyoktmod}) it follows that (\ref{alfakicsi}) remains true.
We note that we did not really need the Economic Cap Covering Theorem in the case of the
ball, however, in the general case we make a full use of it. Naturally, the stated equalities
on the volumes of the caps are not true any more but the existence of $\gamma$ and $\Gamma$
implies that they hold with $\approx$. The calculations finishing the proof can be done
precisely the same way as for the unit ball.

\section{Proof of Theorem~\ref{largenum}}

We are going to show that the asymptotic formula (\ref{poscurv}) and  Theorem~\ref{main} yield
the strong law of large numbers for $V_s(K_n)$ by standard arguments.

We deduce by Chebyshev's inequality that
\begin{eqnarray*}
\mathbb{P}
\left(\left|V_s(K)-V_s(K_n)-\E(V_s(K)-V_s(K_n))\right|n^{\frac2{d+1}}
\geq \varepsilon\right)&\leq &  \varepsilon^{-2}n^{\frac4{d+1}}
{\rm Var}V_s(K_n)\\
&\ll&n^{-\frac{d-1}{d+1}}.
\end{eqnarray*}
Since the sum $\sum_{k=2}^\infty n_k^{-\frac{d-1}{d+1}}$ is finite
for $n_k=k^4$, the sum of the probabilities
$$
\mathbb{P}
\left(\left|V_s(K)-V_s(K_{n_k})-\E(V_s(K)-V_s(K_{n_k}))\right|n_k^{\frac2{d+1}}
\geq \varepsilon\right)
$$
for $k\geq 2$ is finite as well. Therefore the Borel-Cantelli lemma
and the asymptotic formula (\ref{poscurv}) yield that
\begin{equation}
\label{strongsub} \lim_{k\to\infty}(V_s(K)-V_s(K_{n_k}))
n_k^{\frac2{d+1}}= c_{d,j}\int_{S} (\sigma_{d-1}(x))^{\frac{1}{d+1}}
\sigma_{d-j}(x)dx
\end{equation}
with probability $1$. Now, $V_s(K)-V_s(K_n)$ is decreasing, and
hence
$$
(V_s(K)-(V_S(K_{n_{k-1}}))n_{k-1}^{\frac2{d+1}}\leq
(V_s(K)-V_s(K_n))n^{\frac2{d+1}}\leq
(V_s(K)-V_s(K_{n_k}))n_k^{\frac2{d+1}}
$$
hold for $n_{k-1}\leq n\leq n_k$. As
$\lim_{k\to\infty}\frac{n_k}{n_{k-1}}=1$,
the subsequence limit theorem yields Theorem~\ref{largenum}.

\small

\noindent Imre B\'ar\'any\\
Alfr\'ed R\'enyi Institute of Mathematics,\\
PO Box 127, H-1364 Budapest, Hungary, \\
barany@renyi.hu\\
and\\
Department of Mathematics, University College London, \\
Gower Street, London, WC1E 6BT, U.K.\\

\noindent Ferenc Fodor\\
Department of Geometry, University of Szeged,\\
Aradi v\'ertan\'uk tere 1, H-6720 Szeged, Hungary\\
fodorf@math.u-szeged.hu\\
and\\
Department of Mathematics and Statistics\\
University of Calgary\\
2500 University Dr. N.W.\\
Calgary, Alberta, Canada\\
T2N 1N4\\

\noindent Viktor V\'{\i}gh\\
Bolyai Institute, University of Szeged,\\
Aradi v\'ertan\'uk tere 1, H-6720 Szeged, Hungary\\
vigvik@math.u-szeged.hu

\end{document}